\documentclass[english, 11pt]{amsart}
\usepackage[english]{babel}
\usepackage{amssymb, amsmath}
\usepackage[T1,T5]{fontenc}
\newtheorem{theorem}{Theorem}[section]

\newtheorem{corollary}[theorem]{Corollary}

\theoremstyle{definition}
\newtheorem{definition}[theorem]{Definition}
\newtheorem{example}[theorem]{Example}

\numberwithin{equation}{section}

\newcommand{\bR}{\mbox{$\mathbb R$}}

\begin{document}
	\author{Do Thai Duong}
	\address{Institute of Mathematics\\
		Vietnam Academy of Science and Technology \\ 
		18, Hoang Quoc Viet, Hanoi, Vietnam} 
	\email{dtduong@math.ac.vn}
	\title{A COMPARISON THEOREM FOR SUBHARMONIC FUNCTIONS}
	\maketitle	
	\begin{abstract}
		In this article, we prove an extension of the mean value theorem and a comparison theorem for subharmonic functions. These theorems are used to answer the question whether we can conclude that two subharmonic functions which agree almost everywhere on a surface with respect to the surface measure must coincide everywhere on that surface. We prove that this question has a positive answer in the case of hypersurfaces, and we also provide a counterexample in the case of surfaces of higher co-dimension. We also apply these results to Ahlfors-David sets and we prove other versions of the main results in terms of measure densities.
	\end{abstract}
	\section{Introduction}
	Throughout this paper, we always assume that $\Omega$ is a domain of $\mathbb{R}^n$ $(n\geq2)$. Let  $\mathbb{B}(x,r)$ and $\overline{\mathbb{B}}(x,r)$ respectively denote the open ball and the closed ball in $\mathbb{R}^n$, with center at $x\in\bR^n$ and radius $r>0$. The letters $\lambda$ and $\sigma$ will be used to denote respectively the Lebesgue measure and the surface measure in any dimension and on any surface (the context will always clarify their domains of definition).\\
	According to~\cite{Klimek}, we define the subharmonic function on $\Omega$ as following. Let $u:\Omega\longrightarrow[-\infty,+\infty)$ be an upper semicontinuous function which is not identically $-\infty$. Such a function $u$ is said to be subharmonic if for every relatively compact open subset $G$ of $\Omega$ and every function $\varphi\in\mathcal{H}(G)\cap\mathcal{C}(\bar{G})$, the following implication is true:
	\begin{center}
		$u\leq \varphi$ on $\partial G$ $\Longrightarrow$ $u\leq \varphi$ on $G$,
	\end{center}
	where $\mathcal{H}(G)$ is set of all harmonic functions on $G$ and $\mathcal{C}(\bar{G})$ is set of all continuous functions on $\bar{G}$. In this case, we write $u\in\mathcal{SH}(\Omega)$. It is well-known that if $u,v\in \mathcal{SH}(\Omega)$ and $u= v$ almost everywhere on $\Omega$ with respect to the Lebesgue measure, then $u= v$ on $\Omega$ (see e.g.~\cite{Gar},~\cite{Ho},~\cite{Klimek}).\\
	Our work focuses on extending the above result by considering the comparison of two subharmonic functions on a Borel set, with respect to a Borel measure that satisfies some certain conditions. Here is our first result:
	
	\bigskip
	\noindent \textbf{Main Theorem 1.} (Extension of the mean value theorem)
	\textit{Let $u$ be a subharmonic function in a domain $\Omega$, in $\mathbb{R}^n$ $(n\geq 2)$,
		$K$ be a Borel subset of $\Omega$ and $x_0$ be a point of $K$. Let $h:(0,+\infty)\rightarrow (0,+\infty)$ be a function such that there are real numbers $M>0$ and $c>4$ satisfying:
		$$\int\limits_0^{c\epsilon}\frac{h(r)}{r^{n-1}}dr\leq M \frac{h(\epsilon)}{\epsilon^{n-2}},$$
		for $\epsilon$ small enough. Suppose that there exist a positive Borel measure $\mu$, real numbers $A,B>0$ and $\epsilon_0>0$ satisfying the following:
		\begin{itemize}
			\item[1.] $\mu\big(K\cap \mathbb{B}(x_0,\epsilon)\big)\geq Ah(\epsilon)$ for all $\epsilon<\epsilon_0$,
			\item[2.] $\mu\big(K\cap \mathbb{B}(x,\epsilon)\big)\leq Bh(\epsilon)$ for all $x\in K$ and $\epsilon<\epsilon_0$.
		\end{itemize}	
		Then
		$$\lim\limits_{\epsilon\rightarrow 0}\frac{1}{\mu\big(K\cap\mathbb{B}(x_0,\epsilon)\big)}\int\limits_{K\cap\mathbb{B}(x_0,\epsilon)}u(x)d\mu(x)=u(x_0).$$}
	\bigskip
	
	We use the Main Theorem 1 to prove the second result.
	
	\bigskip
	\noindent\textbf{Main Theorem 2.} (Comparison theorem for subharmonic functions)
	\textit{Let $u$ be an upper semicontinuous function, $v$ be a subharmonic function in a domain $\Omega$, in $\mathbb{R}^n$ $(n\geq 2)$ and let $K$ be a Borel subset of $\Omega$. Let $h:(0,+\infty)\rightarrow (0,+\infty)$ be a function such that there are real numbers $M>0$ and $c>4$ satisfying:
		$$\int\limits_0^{c\epsilon}\frac{h(r)}{r^{n-1}}dr\leq M\frac{h(\epsilon)}{\epsilon^{n-2}},$$
		for $\epsilon$ small enough. Suppose that there exist a positive Borel measure $\mu$, real numbers  $A,B>0$, $\epsilon_0>0$ and $N\subset K$ satisfying the following:
		\begin{itemize}
			\item[1.] $\mu(N)=0,$
			\item[2.] $Ah(\epsilon)\leq\mu\big(K\cap \mathbb{B}(x,\epsilon)\big)\leq Bh(\epsilon)$ for all $x\in K$ and $\epsilon<\epsilon_0$,
			\item[3.] $u\geq v$ on $K\diagdown N$.
		\end{itemize}
		Then $u\geq v$ on $K$.}
	
	\bigskip
	\noindent\textbf{Remark 1.}\textit{ The family of admissible functions,
		\begin{multline*}
			\mathcal{F}=\Bigg\{h:(0,+\infty)\rightarrow (0,+\infty):\, \exists\ M,\epsilon_0>0,\, c>4\, \text{such that}\\
			\int\limits_0^{c\epsilon}\frac{h(r)}{r^{n-1}}dr\leq M\frac{h(\epsilon)}{\epsilon^{n-2}},\, \forall\, \epsilon\in (0,\epsilon_0]. \Bigg\},
		\end{multline*}
		contains many functions such as $r^k$, $r^k|\log r|$ where $k>n-2$. If $h_1,\, h_2\in \mathcal{F}$ then $h_1+h_2\in\mathcal{F}$ and $a\cdot h_1\in\mathcal{F}$ for every positive number $a$. Here is one simple way of interpreting the
		admissible functions $h$. For a function $F:(0,+\infty)\rightarrow (0,+\infty)$, we could
		consider the following property:
		\begin{align}
			\exists c>4\ \text{such that}\ \limsup\limits_{\epsilon\rightarrow0^+}\frac{\frac{1}{c\epsilon}\int\limits_0^{c\epsilon}F(r)dr}{F(\epsilon)}<\infty.\tag{$*$}
		\end{align}
		This can be viewed as a very weak asymptotic one-dimensional mean value property as we are comparing integral means with the value of the integrated function inside the interval of integration. A function $h$ belongs to $\mathcal{F}$ if and only if $F(r)=h(r)/r^{n-1}$ satisfies $(*)$.
	}
	
	\bigskip
	\noindent\textbf{Remark 2.} \textit{When K is appropriate and $h$ is a gauge (i.e. $h:[0,\infty)\rightarrow[0,\infty)$, h is non-decreasing, right-continuous and equal to 0 only at 0), we can apply the main theorems for $\mu$ as Generalized Hausdorff $h$-measures. In particular, if we choose $h(r)=r^k$, the second condition in Main Theorem 2 becomes the Ahlfors-David condition, $\mu$ can be a measure obtained from the classic Caratheodory's construction such as the $k$-dimensional Hausdorff measure, the $k$-dimensional spherical measure or the $k$-dimensional net measure (see~\cite{Mattila} for more information on these measures).}
	
	\bigskip
	From these results, we can conclude that two subharmonic functions which agree almost everywhere on a hypersurface with respect to the surface measure must coincide everywhere on that hypersurface. By constructing a counterexample, it is
	also shown that this property may fail for surfaces of higher co-dimension. We also apply these main results to Ahlfors-David regular sets which have been investigated in different situations such as in connection with some function spaces or in complex and harmonic analysis. In the last section, we prove another versions of the main results in terms of measure densities.
	\section{Preliminaries}
	In this section, we recall some definitions and results that will be used in the sequel.
	Throughout our work, we will use the Riesz Decomposition theorem as a major technical tool. 
	\begin{theorem}[see~\cite{Klimek} and~\cite{Rans}]
		Suppose that $u$ is a subharmonic function in a domain $\Omega$ in $\mathbb{R}^n$ ($n\geq2$). Given a relatively compact open subset $U$ of $\Omega$, we can decompose $u$ as
		$$u(x)=\frac{-1}{\max\{1,n-2\}\sigma\big(\partial \mathbb{B}(0,1)\big)}\int\limits_{U}g(|x-w|)d\nu(w)+\varphi(x),$$
		on $U$ where $\nu=\Delta u|_{U}$, $\varphi\in \mathcal{H}(U)$ and the kernel $g:(0,+\infty)\rightarrow\bR$ is defined by
	\begin{equation}\label{equa_1}
	g(r)=
	\begin{cases}
	-\log r & $(n=2)$ \\
	r^{2-n} & $(n>2)$
	\end{cases}.
	\end{equation}
	
	\end{theorem}
	The next classical result is useful for proofs of main theorems. 
	\begin{theorem}[see Theorem 1.15 in~\cite{Mattila}]\label{theo1}
		Let $\mu$ be a Borel measure and $f$ be a non-negative Borel function on a separable metric space $X$. Then
		$$\int\limits_X fd\mu=\int\limits_0^{+\infty}\mu(\{x\in X:f(x)\geq t\})dt. $$
	\end{theorem}
	Next, we present the definitions of the $k$-dimensional Hausdorff measure, the Ahlfors-David regular set, the upper and lower densities of a Radon measure.
	\begin{definition}[see~\cite{Mattila}]
	Let $A\subset\bR^n$, we define:
	$$H^k_\delta(A)=\inf\{\sum_i d(E_i)^k: A\subset\bigcup_i E_i, d(E_i)\leq\delta\},$$ where $d(E)$ is the diameter of $E$:
	$$d(E)=\sup\limits_{x,y\in E}|x-y|.$$
	The $k$-dimensional Hausdorff measure of $A$, denoted by $H^k(A)$, is defined by
		$$H^k(A)=\lim\limits_{\delta\downarrow0}H^k_\delta(A).$$
	\end{definition}
	\begin{definition}[see~\cite{David},~\cite{Mattila09} and~\cite{Wang}]
		A subset $E$ of $\mathbb{R}^n$ is said to be Ahlfors-David regular with dimension $k$ if it is closed and if there is a constant $C_0>0$ such that
		$$C_0^{-1} R^k\leq H^k(E\cap \mathbb{B}(x,R))\leq C_0 R^k, $$
		for all $x\in E$ and $0<R<d(E)$.
	\end{definition}
	\begin{definition}[see~\cite{Mattila}]
		Let $0\leq s<\infty$ and let $\eta$ be a Radon measure on $\mathbb{R}^n$. The upper and lower $s$-densities of $\eta$ at $x\in\mathbb{R}^n$ are defined by
		$$\Theta^{*s}(\eta,x)=\limsup\limits_{r\downarrow0}\frac{\eta\big(\mathbb{B}(x,r)\big)}{(2r)^s},$$
		$$\Theta_*^s(\eta,x)=\liminf\limits_{r\downarrow0}\frac{\eta\big(\mathbb{B}(x,r)\big)}{(2r)^s}.$$
	\end{definition}
	We end this preparatory section by presenting a notation which will be used in the last section.
		For a Borel measure $\eta$ on $\mathbb{R}^n$ and a Borel set $K$, we define:
		$$\eta_K(E)=\eta(K\cap E).$$
	It is clear that $\eta_K$ is also a Borel measure on $\mathbb{R}^n$.
	\section{proof of main results}
	\begin{proof}[Proof of Main Theorem 1]
	It is sufficient to prove the theorem when $x_0=0$. We choose $\epsilon_1<\frac{\epsilon_0}{c}$ such that the inequality in the condition of $h$ holds for $\epsilon\leq\epsilon_1$ and we choose $\gamma>1$, $p>1$ such that $c=2\gamma(1+p)$. For convenience, we set $K_\epsilon=K\cap\mathbb{B}(0,\epsilon)$. Our goal is to establish:
		\begin{equation}\label{equa_2}
			\lim\limits_{\epsilon\rightarrow 0}\frac{1}{\mu(K_\epsilon)}\int\limits_{K_\epsilon}u(x)d\mu(x)=u(0).
		\end{equation}
		If $u(0)=-\infty$, then we have
		$$\limsup\limits_{\epsilon\rightarrow0}\frac{1}{\mu(K_\epsilon)}\int\limits_{K_\epsilon}u(x)d\mu(x)\leq\limsup\limits_{\epsilon\rightarrow0}\sup\limits_{y\in K_\epsilon}u(y)\\
			\leq u(0)=-\infty.$$
		Hence, $$\lim\limits_{\epsilon\rightarrow0}\frac{1}{\mu(K_\epsilon)}\int\limits_{K_\epsilon}u(x)d\mu(x)=u(0).$$
		We now turn to the case where $u(0)>-\infty$. Since the problem is local, we can assume that $\mathbb{B}(0,1)\Subset\Omega$. By the Riesz Decomposition theorem, we can write
		\begin{equation}\label{equa_7}
			u(x)=\frac{-1}{\max\{1,n-2\}\sigma\big(\partial \mathbb{B}(0,1)\big)}\int\limits_{\mathbb{B}(0,1)}g(|x-w|)d\nu(w)+\varphi(x),
		\end{equation}
		on $\mathbb{B}(0,1)$, where $\nu=\Delta u|_{\mathbb{B}(0,1)}$ is a Radon nonnegative measure, $\varphi\in \mathcal{H}\big(\mathbb{B}(0,1)\big)$ and the kernel $g$ is defined by (\ref{equa_1}).	
		It thus follows from equality (\ref{equa_7}), Fubini's theorem and the continuity of $\varphi$ that
		\begin{equation}\label{equa_5}
		\begin{split}
		&\lim\limits_{\epsilon\rightarrow 0}\frac{1}{\mu(K_\epsilon)}\int\limits_{K_\epsilon}u(x)d\mu(x)=\\
		&\lim\limits_{\epsilon\rightarrow 0}\frac{-1}{\max\{1,n-2\}\sigma\big(\partial \mathbb{B}(0,1)\big)}\int\limits_{ \mathbb{B}(0,1)}f_\epsilon(w)d\nu(w)+\varphi(0),
		\end{split}
		\end{equation}
		where $$f_\epsilon(w)=\frac{1}{\mu(K_\epsilon)}\int\limits_{K_\epsilon}g(|x-w|)d\mu(x).$$
		By replacing $x$ by $0$ in equality (\ref{equa_7}), we have
		\begin{equation}\label{equa_6}
		u(0)=\frac{-1}{\max\{1,n-2\}\sigma\big(\partial \mathbb{B}(0,1)\big)}\int\limits_{\mathbb{B}(0,1)}g(|w|)d\nu(w)+\varphi(0).
		\end{equation}
		By (\ref{equa_5}) and (\ref{equa_6}), the equality (\ref{equa_2}) is equivalent to
		\begin{equation}\label{equa_4}
			\lim\limits_{\epsilon\rightarrow0}\int\limits_{ \mathbb{B}(0,1)}f_\epsilon(w)d\nu(w)=\int\limits_{ \mathbb{B}(0,1)}g(|w|)d\nu(w).
		\end{equation}
		We prove (\ref{equa_4}) by three steps:\\
		\textbf{Step 1:} We claim that
		$$f_\epsilon(w)\rightarrow g(|w|),\ \epsilon\rightarrow0,$$ almost everywhere on $\mathbb{B}(0,1)$ with respect to $\nu$. Indeed, consider (\ref{equa_6}), since the measure $\nu$ is nonnegative and the function $g$ is positive and decreasing on $(0,1)$, we have
		\begin{align*}
			u(0)&\leq\frac{-1}{\max\{1,n-2\}\sigma\big(\partial \mathbb{B}(0,1)\big)}\int\limits_{\mathbb{B}(0,\delta)}g(|w|)d\nu(z)+\varphi(0)\\
			&\leq\frac{-1}{\max\{1,n-2\}\sigma\big(\partial \mathbb{B}(0,1)\big)}g(\delta)\nu\big(\mathbb{B}(0,\delta)\big)+\varphi(0),
		\end{align*}
	for all $0<\delta<1$. This implies that 
	$$\nu(\mathbb{B}(0,\delta))\leq\max\{1,n-2\}\sigma(\partial \mathbb{B}(0,1))\frac{-u(0)+\varphi(0)}{g(\delta)},$$ 
	for all $0<\delta<1$. By letting $\delta\rightarrow0$, we conclude that $\nu(\mathbb{B}(0,\delta))\searrow0$ when $\delta\searrow0$. Hence $\nu(\{{0\}})=0$. Combining this with the fact that $$f_\epsilon(w)\rightarrow g(|w|),$$ pointwise for all $w\in\mathbb{B}(0,1)\diagdown\{0\}$, we see that the claim holds.\\
	\textbf{Step 2:} We will show that there exist constants $C_1, C_2>0$ such that 
		$$f_\epsilon(w)\leq C_1g(|w|)+C_2,$$ for all $w\in \mathbb{B}(0,1)\diagdown\{0\}$ and $\epsilon<\epsilon_1$. We split the proof into two cases: Case I, where $|w|> p\epsilon$, and Case II, where $|w|\leq p\epsilon$.\\
		We first consider \textbf{Case I}, the case where $|w|> p\epsilon$. We observe that, for $x\in K_\epsilon$,
		$$|x-w|\geq|w|-|x|\geq|w|-\epsilon>|w|-\frac{|w|}{p}=\frac{ p-1}{p}|w|.$$
		Therefore, by the definition of $f_\epsilon$ and the decrease of the function $g$ on $(0,1)$,
		\begin{align*}
			f_\epsilon(w)&\leq \frac{1}{\mu(K_\epsilon)}\int\limits_{K_\epsilon}g(|w|-|x|)d\mu(x)\\
			&\leq\frac{1}{\mu(K_\epsilon)}\int\limits_{K_\epsilon}g\Big(\frac{ p-1}{p}|w|\Big)d\mu(x)\\
			&=g\Big(\frac{ p-1}{ p}|w|\Big).
		\end{align*}
		The proof of Case I is finished by noticing that 
		$$g\Big(\frac{ p-1}{ p}|w|\Big)=
		\begin{cases}
		g(|w|)+g\Big(\frac{ p-1}{ p}\Big) & $(n=2)$\\
		\Big(\frac{ p-1}{ p}\Big)^{2-n}\ g(|w|) & $(n>2)$
		\end{cases}.
		$$
		Next, we study \textbf{Case II}, the case where $|w|\leq p\epsilon$. By Theorem \ref{theo1} and the definition of $f_\epsilon$, we have
		\begin{align*}
			f_\epsilon(w)&=\frac{1}{\mu(K_\epsilon)}\int\limits_0^{+\infty}\mu\Big(\{x\in K_\epsilon: g(|x-w|)\geq t \}\Big)dt\\
			&=\frac{1}{\mu(K_\epsilon)}\int\limits_0^{+\infty}\mu\Big( K_\epsilon\cap \overline{\mathbb{B}}\big(w,g^{-1}(t)\big) \Big)dt
		\end{align*}
		Therefore, by splitting the integral at $\alpha(\epsilon)$ where $\alpha(\epsilon)=g(|w|+\epsilon)$,
			$$f_\epsilon(w)=\frac{1}{\mu(K_\epsilon)} \int\limits_0^{\alpha(\epsilon)}\mu\Big( K_\epsilon\cap \overline{\mathbb{B}}\big(w,g^{-1}(t)\big)\Big)dt + \frac{1}{\mu(K_\epsilon)}\int\limits_{\alpha(\epsilon)}^{+\infty}\mu\Big( K_\epsilon\cap \overline{\mathbb{B}}\big(w,g^{-1}(t)\big)\Big)dt.$$
	For the first term of the right-hand side, it is clear that $\mathbb{B}(0,\epsilon)\subset \mathbb{B}\big(w,g^{-1}(t)\big)$ and hence 
		\begin{equation}\label{ineq_1}
		\mu\Big( K_\epsilon\cap \overline{\mathbb{B}}\big(w,g^{-1}(t)\big)\Big)=\mu(K_\epsilon),\ \text{for all}\ 0\leq t\leq\alpha(\epsilon). 
		\end{equation}
		For the second term of the right-hand side, if $K\cap \mathbb{B}\big(w,\gamma g^{-1}(t)\big)=\emptyset$ then $\mu\Big( K_\epsilon\cap \overline{\mathbb{B}}\big(w,g^{-1}(t)\big)\Big)=0$. Otherwise,  
		$$K_\epsilon\cap \overline{\mathbb{B}}\big(w,g^{-1}(t)\big)\subset K\cap \mathbb{B}\big(w,\gamma g^{-1}(t)\big)\subset K\cap\mathbb{B}\big(w_0,2\gamma g^{-1}(t)\big),$$ where $w_0\in K\cap \mathbb{B}\big(w,\gamma g^{-1}(t)\big)$. Therefore, by the assumption of the theorem and the fact that $2\gamma g^{-1}(t)<\epsilon_0$ for all $t\geq\alpha(\epsilon)$ and $\epsilon<\epsilon_1$, we obtain
		\begin{equation}\label{ineq_2}
		\mu\Big( K_\epsilon\cap \overline{\mathbb{B}}\big(w,g^{-1}(t)\big)\Big)\leq Bh\big(2\gamma g^{-1}(t)\big),\ \text{for all}\ t\geq\alpha(\epsilon).
		\end{equation}		
		Thus, combining (\ref{ineq_1}), (\ref{ineq_2}) with the expression of $f_\epsilon$, we have for any $\epsilon<\epsilon_1$, 
		\begin{align*}
			f_\epsilon(w)
			&\leq\frac{1}{\mu(K_\epsilon)} \int\limits_0^{\alpha(\epsilon)}\mu(K_\epsilon)dt +\frac{B}{\mu(K_\epsilon)} \int\limits_{\alpha(\epsilon)}^{+\infty}h(2\gamma g^{-1}(t)) dt\\
			&=\alpha(\epsilon)+\frac{(2\gamma)^{n-2}\max\{1,n-2\}B}{\mu(K_\epsilon)} \int\limits_{0}^{2\gamma(|w|+\epsilon)}\frac{h(r)}{r^{n-1}} dr\\
			&\leq g\big(|w|+\epsilon\big)+\frac{(2\gamma)^{n-2}\max\{1,n-2\}B}{Ah(\epsilon)}\int\limits_{0}^{2\gamma(p+1)\epsilon}\frac{h(r)}{r^{n-1}} dr\\
			&\leq g(|w|)+\frac{(2\gamma)^{n-2}\max\{1,n-2\}BM}{A\epsilon^{n-2}},
		\end{align*}
		the last inequality is deduced from the assumption of $h$. Furthermore, by the assumption of Case II, we get
		\begin{equation}\label{equa_3}
		f_\epsilon(w)\leq g(|w|)+\frac{(2\gamma p)^{n-2}\max\{1,n-2\}BM}{A|w|^{n-2}},
		\end{equation}
		for all $\epsilon<\epsilon_1$. It is clear that the second term in the right-hand side of (\ref{equa_3}) is a constant when $n=2$ and is equal to
		$$\frac{(2\gamma p)^{n-2}(n-2)BM}{A}g(|w|)$$
		when $n>2$. The proof of Case II and, therefore, of Step 2 is complete.\\
		\textbf{Step 3:} By conclusions of Step 1, Step 2 and the Lebesgue's Dominated Convergence theorem, we derive (\ref{equa_4}), which completes the proof.
	\end{proof}
	Using the extension of the mean value theorem, we now ready to prove the second main result.
	\begin{proof}[Proof of Main Theorem 2]
		Let $x_0\in K$, it is sufficient to show that $u(x_0)\geq v(x_0)$.
		By the upper semicontiniuty of $u$, we have
			$$u(x_0)\geq\lim\limits_{\epsilon\rightarrow0}\frac{1}{\mu\big(K\cap\mathbb{B}(x_0,\epsilon)\big)}\int\limits_{K\cap\mathbb{B}(x_0,\epsilon)}u(x)d\mu(x).$$
			Following Theorem 3.1, we get 
		$$v(x_0)=\lim\limits_{\epsilon\rightarrow0}\frac{1}{\mu\big(K\cap\mathbb{B}(x_0,\epsilon)\big)}\int\limits_{K\cap\mathbb{B}(x_0,\epsilon)}v(x)d\mu(x).$$
		Since $u\geq v$ almost everywhere on $K$ with respect to $\mu$, we infer that for every $\epsilon>0$, $$\int\limits_{K\cap\mathbb{B}(x_0,\epsilon)}u(x)d\mu(x)\geq\int\limits_{K\cap\mathbb{B}(x_0,\epsilon)}v(x)d\mu(x).$$
		Combining the above inequalities gives $u(x_0)\geq v(x_0)$, as desired.
	\end{proof}
	\section{Some Consequences and a counterexample}
	By applying the main results in the case where $K$ is a hypersurface, $\mu$ is the surface measure on $K$ and $h(r)=r^{n-1}$, we obtain the following immediate corollaries:
	\begin{corollary}
		Let $u$ be a subharmonic function in a domain $\Omega$, in $\mathbb{R}^n$ $(n\geq 2)$ and $\mathbb{H}$ be a hypersurface. Then
		$$\lim\limits_{\epsilon\rightarrow 0}\frac{1}{\sigma\big(\mathbb{H}\cap\mathbb{B}(x_0,\epsilon)\big)}\int\limits_{\mathbb{H}\cap\mathbb{B}(x_0,\epsilon)}u(x)d\sigma(x)=u(x_0)$$ for all $x_0\in \mathbb{H}\cap\Omega$, where $\sigma$ is the surface measure on $\mathbb{H}$.
	\end{corollary}
	\begin{corollary}
		Let $\Omega$ be a domain in $\mathbb{R}^n$ $(n\geq 2)$ and let $\mathbb{H}$ be a hypersurface such that $\mathbb{H}\cap\Omega\neq\emptyset$. Let $u$ be an upper semicontinuous function and $v$ be a subharmonic function in $\Omega$. Suppose that $u\geq v$ almost everywhere on $\mathbb{H}\cap\Omega$ with respect to the surface measure on $\mathbb{H}$. Then $u\geq v$ on $\mathbb{H}\cap\Omega$.
	\end{corollary}
	Next, we construct a counterexample to show that the comparison theorem for subharmonic functions will be false if we consider $K$ as a surface of dimension $k\leq n-2$, $\mu$ as the surface measure and $h(r)=r^k$. This means that the Corollary 4.2 is no longer true when $\mathbb{H}$ is a surface of dimension $k\leq n-2$.
	\begin{example}[Counterexample]
		In $\mathbb{R}^n$ $(n\geq3)$, we denote by $\mathbb{B}_{n-2}(0,R)$ the open ball in $\bR^{n-2}$, with center at 0 and radius $R>0$. 
		For $i\geq2$, let $\mu_i$ be the measure defined on $\big(\mathbb{B}_{n-2}(0,i)\diagdown\mathbb{B}_{n-2}(0,\frac{1}{i})\big)\times\{0\}\times\{0\}$ generated by the
		 $(n-2)$-dimensional Lebesgue measure on $\big(\mathbb{B}_{n-2}(0,i)\diagdown\mathbb{B}_{n-2}(0,\frac{1}{i})\big)$. We define potentials $p_{\mu_i}:\mathbb{R}^n\longrightarrow[-\infty,\infty)$ by
		$$p_{\mu_i}(x)=\int\limits_{\mathbb{\mathbb{R}}^n}\frac{-1}{|x-w|^{n-2}}d\mu_i(w).$$
		Then we obtain the sequence $\{p_{\mu_i}\}_{i\geq2}\subset\mathcal{SH}(\mathbb{R}^n)$ which satisfies these properties:\\
		$$\begin{cases}
		p_{\mu_i}\leq0\ \text{on}\ \mathbb{R}^n,\\
		p_{\mu_i}=-\infty\ \text{on}\ \big(\mathbb{B}_{n-2}(0,i)\diagdown\mathbb{B}_{n-2}(0,\frac{1}{i})\big)\times\{0\}\times\{0\},\\
		-\infty<p_{\mu_i}(0)<0,
		\end{cases}$$
		for all $i\geq2$. By setting $$u_i(x)=-\frac{p_{\mu_i}(x)}{p_{\mu_i}(0)},$$ the sequence $\{u_i\}_{i\geq2}\subset\mathcal{SH}(\mathbb{R}^n)$ has these properties:
		$$\begin{cases}
		u_i\leq0\ \text{on}\ \mathbb{R}^n,\\
		u_i=-\infty\ \text{on}\ \big(\mathbb{B}_{n-2}(0,i)\diagdown\mathbb{B}_{n-2}(0,\frac{1}{i})\big)\times\{0\}\times\{0\},\\
		u_i(0)=-1,
		\end{cases}$$
		for all $i\geq2$. Now we define  
		$$u(x)=\sum_{i=2}^\infty \frac{1}{2^{i-1}}u_i(x),$$
		hence
		$$\begin{cases}
		u\in\mathcal{SH}(\mathbb{R}^n),\\
		u=-\infty\ \text{on}\ \big(\mathbb{R}^{n-2}\diagdown\{0\}\big)\times\{0\}\times\{0\},\\
		u(0)=-1.
		\end{cases}$$
		Therefore, the function $\widetilde{u}=\max(u,-2)$ satisfies:
		$$\begin{cases}
		\widetilde{u} \in\mathcal{SH}(\mathbb{R}^n),\\
		\widetilde{u}=-2\ \text{on}\ \big(\mathbb{R}^{n-2}\diagdown\{0\}\big)\times\{0\}\times\{0\},\\ \widetilde{u}(0)=-1.
		\end{cases}$$
		Finally, by setting $\widetilde{v}\equiv-2$, we conclude that $\widetilde{v}\geq \widetilde{u}$ almost everywhere on $\mathbb{R}^{n-2}\times\{0\}\times\{0\}$ with respect to $(n-2)$-dimensional the Lebesgue measure on $\mathbb{R}^{n-2}\times\{0\}\times\{0\}$, but not everywhere as $\widetilde{v}(0)<\widetilde{u}(0)$.
	\end{example}
	Next, we apply the main theorems to Ahlfors-David regular sets. Considering $h(r)=r^k$ where $k>n-2$, these corollaries are direct consequences of the main results.
	\begin{corollary}
		Let $u$ be a subharmonic function in a domain $\Omega$, in $\mathbb{R}^n$ $(n\geq 2)$ and $E\subset\Omega$ be an Ahlfors-David regular set with dimension $k>n-2$. Then
		$$\lim\limits_{\epsilon\rightarrow 0}\frac{1}{H^k(E\cap\mathbb{B}(x,\epsilon))}\int\limits_{E\cap\mathbb{B}(x,\epsilon)}u(y)dH^k(y)=u(x)$$ for all $x\in E$.
	\end{corollary}
	\begin{corollary}
		Let $\Omega$ be a domain in $\mathbb{R}^n$ $(n\geq 2)$ and $E\subset\Omega$ be an Ahlfors-David regular set with dimension $k>n-2$. Let $u$ be an upper semicontinuous function and $v$ be a subharmonic function in $\Omega$. Suppose that $u\geq v$ almost everywhere on $E$ with respect to $k$-dimensional Hausdorff measure. Then $u\geq v$ on $E$.
	\end{corollary}
	\section{Other versions of main results}
	The main idea of the proof of Main Theorem 1 is that the functions $f_\epsilon$ is bounded by integrable functions with respect to $\nu$. By upper and lower densities of a measure, this theorem below is another version of Main Theorem 1.
	\begin{theorem}
		Let $\Omega$ be a domain in $\mathbb{R}^n$ $(n\geq 2)$. Let $K$ be a Borel subset of $\Omega$ and $x_0$ be a point of $K$. Suppose that there exist a positive Borel measure $\mu$, a relatively compact open subset $U$ of $\Omega$ that contains $x_0$ and a positive number $s>n-2$ satisfying the following:
		$$\frac{1}{\Theta_*^s(\mu_K,x_0)}\int\limits_U\frac{\Theta^{*s}(\mu_K,w)}{|w|^{n-2}}d\nu(w)<+\infty$$
		where $\nu=\Delta u|_U$.
		Then
		$$\lim\limits_{\epsilon\rightarrow 0}\frac{1}{\mu\big(K\cap\mathbb{B}(x_0,\epsilon)\big)}\int\limits_{K\cap\mathbb{B}(x_0,\epsilon)}u(x)d\mu(x)=u(x_0).$$	
	\end{theorem}
	\begin{proof}
		We can assume that $x_0=0$ and $U=\mathbb{B}(0,1)$. Adapting to the technique used in the proof of Main Theorem 1, it remains to show that $f_\epsilon$ is bounded from above by integrable function with respect to $\nu$. Under the assumption of densities, it is sufficient to show that there exist constants $C_1, C_2, C_3$ such that 
		$$f_\epsilon(w)\leq C_1g(|w|)+\frac{C_2}{\Theta_*^s(\mu_K,0)}.\frac{\Theta^{*s}(\mu_K,w)}{|w|^{n-2}} +C_3,$$ for all $w\in \mathbb{B}(0,1)\diagdown\{0\}$ and $\epsilon>0$ small enough. For $p>1$, we also consider two cases as before. The proof differs only in the case II where $|w|\leq p\epsilon$.
		In this case, we have
		\begin{align*}
			f_\epsilon(w)&\leq\frac{1}{\mu(K_\epsilon)} \int\limits_0^{\alpha(\epsilon)}\mu(K_\epsilon)dt +\frac{1}{\mu(K_\epsilon)} \int\limits_{\alpha(\epsilon)}^{+\infty}\mu( K\cap \mathbb{B}(w,2g^{-1}(t))) dt\\
			&\leq\alpha(\epsilon) +\frac{1}{(2\epsilon)^s\Theta_*^s(\mu_K,0)} \int\limits_{\alpha(\epsilon)}^{+\infty}(4g^{-1}(t))^s\Theta^{*s}(\mu_K,w) dt\\
			&\leq g(|w|)+\frac{2^s\max\{1,n-2\}}{s-(n-2)}.\frac{\Theta^{*s}(\mu_K,w)}{\Theta_*^s(\mu_K,0)}.\Big(\frac{|w|+\epsilon}{\epsilon}\Big)^s.\frac{1}{(|w|+\epsilon)^{n-2}}\\
			&\leq
			g(|w|)+\frac{2^s.(p+1)^s\max\{1,n-2\}}{s-(n-2)}.\frac{\Theta^{*s}(\mu_K,w)}{\Theta_*^s(\mu_K,0)}.\frac{1}{|w|^{n-2}},	
		\end{align*}
		as desired.
	\end{proof}
	The next result, as a consequence of Theorem 5.1, is another version of Main Theorem 2. Their proofs are the same.
	\begin{theorem}
		Let $\Omega$ be a domain in $\mathbb{R}^n$ $(n\geq 2)$ and $K$ be a Borel subset of $\Omega$. Let $u$ be an upper semicontinuous function and $v$ be a subharmonic function in $\Omega$. Suppose that there exist a positive Borel measure $\mu$ and a positive number $s>n-2$ such that for all $x\in K$, there exists a relatively compact open subset $U_x$ of $\Omega$ that contains $x$ satisfying:
		$$\frac{1}{\Theta_*^s(\mu_K,x)}\int\limits_{U_x}\frac{\Theta^{*s}(\mu_K,w)}{|w|^{n-2}}d\nu(w)<+\infty$$
		where $\nu=\Delta v|_{U_x}$. If $u\geq v$ almost everywhere on $K$ with respect to $\mu$ then $u\geq v$ on $K$.
	\end{theorem}
	\vskip1cm
	
	{\bf Acknowledgment.}  This work forms part of the author's doctoral dissertation under the supervision of Professor Dinh Tien Cuong and Professor Pham Hoang Hiep. The author is grateful to receive valuable comments and strong support from his advisors, Dr. Do Hoang Son and the referees. The author would like to thank IMU and TWAS for supporting his PhD studies through the IMU Breakout Graduate Fellowship.

\end{document}